\documentclass[11pt]{amsart}

\usepackage{amsmath,amssymb,amsthm,mathrsfs}
\usepackage{mycommands}
\usepackage{tikz}
\usetikzlibrary{arrows}
\usepackage[colorinlistoftodos, textsize=tiny]{todonotes}

\title{Berge duals and universally tight contact structures}
\author{Christopher R. Cornwell}
\address{CIRGET \\ Universit\'e du Qu\'ebec \`a Montr\'eal, Montreal, QC}
\email{cornwell@cirget.ca}

\makeatletter
\providecommand\@dotsep{5}
\def\listtodoname{List of Todos}
\def\listoftodos{\@starttoc{tdo}\listtodoname}
\makeatother

\begin{document}


\begin{abstract}Dehn surgery on a knot determines a \emph{dual} knot in the surgered manifold, the core of the filling torus. We consider duals of knots in $S^3$ that have a lens space surgery. Each dual supports a contact structure. We show that if a universally tight contact structure is supported, then the dual is in the same homology class as the dual of a torus knot.
\end{abstract}

\maketitle


\section{Introduction}
\label{sec:intro}
Let $p,q$ be coprime with $0\le q<p$. Considering $S^3$ as the unit sphere of $\C^2$, with coordinates $(z_1,z_2)$, define $L(p,q)$ as the quotient of $S^3$ under the $(\Z/p\Z)$-action $(z_1,z_2) \sim (\omega z_1,\omega^qz_2)$ where $\omega=\exp(2\pi\sqrt{-1}/p)$. Write $\pi:S^3\to L(p,q)$ for the associated covering map. 

If there is a Dehn surgery on a knot $K\subset S^3$ producing some lens space $L(p,q)$ then $K$ is called a \emph{lens space knot}. Passing to the mirror of $K$ if needed, we consider only positive surgeries. A method of constructing lens space knots is given in \cite{Ber}: let $\Sigma$ give a genus 2 Heegaard splitting $S^3 = H_1\cup_{\Sigma}H_2$; any knot embedded in $\Sigma$ so that it is primitive in $\pi_1(H_i)$ (under the inclusion $\Sigma\subset H_i$, $i=1,2$), has a lens space surgery with framing given by $\Sigma$. Knots which can be so embedded are called \emph{doubly primitive}. The Berge Conjecture \cite[Problem 1.78]{Kir95} is that every lens space knot is doubly primitive.

Berge catalogued ten families of doubly primitive knots which we refer to as Types I through X. \footnote{In fact, Berge described twelve families, I{--}XII. As in \cite{Ras}, our description in Types IX and X account for those in XI and XII by allowing the parameter $j$ to be negative.} Following convention we call the knots in Types I through X the \emph{Berge knots}, though it is now known that every doubly primitive knot is one of the Berge knots \cite{Gre}.

Recall that Dehn surgery on $K\subset S^3$ picks out a knot in the surgered manifold, the \emph{surgery dual}, defined as the core of the surgery solid torus. As elsewhere in the literature, if $K\subset S^3$ is a Berge knot and $K'\subset L(p,q)$ its surgery dual, we also call $K'\subset L(p,q)$ a Berge knot. A number of investigations take this dual perspective (e.g.\ \cite{BGH},\cite{Gre},\cite{Hed2},\cite{Ras}). 

By the work of numerous authors, a dual in $L(p,q)$ to a lens space knot is rationally fibered and supports a tight contact structure on $L(p,q)$ (see discussion in Section \ref{sec:bg}). A universally tight contact structure on $L(p,q)$ is obtained by using $\pi$ to push forward the standard contact structure $\xi_{st}$. Denote the resulting contact structure on $L(p,q)$ by $\xi_{p,q}$. When $0<q<p-1$ there is another universally tight contact structure, obtained by reversing the coorientation of the contact planes.  We consider which duals to a lens space knot support one of these contact structures.

\begin{thm} If $K\subset L(p,q)$ is dual to a lens space knot and $K$ supports a universally tight contact structure, then the homology class of $K$ in $H_1(L(p,q))$ contains a Berge knot that is dual to a torus knot. 
\label{mainthm}
\end{thm}

\begin{rem} We will prove (see Theorem \ref{mainthmBerge}) that a Berge knot $B\subset L(p,q)$ that supports a universally tight structure is dual to a torus knot. However, our techniques use only the homology class of $B$ and its genus. That the homology class of $K$ contains a Berge knot $B$ which has the same knot Floer homology as $K$, and thus genus, was shown in \cite{Gre} (see also \cite[Theorem 2]{Ras}).
\end{rem}

\begin{rem}
Theorem \ref{mainthm} implies, for a dual $K$ to a lens space knot, that the natural extension of the transverse invariant $\lambda^+$ of \cite{OST} to the setting of lens spaces does not live in the extremal Alexander grading of $\widehat{HFK}(L(p,q),K)$ unless $K$ is homologous to a torus knot dual.
\end{rem}

\comment{
Our motivation for considering which Berge knots support $\xi_{p,q}$ stems from a possible contact geometric approach to the Berge Conjecture, which we now outline.
 Let $L(p,q)=V_\alpha\underset{\Sigma}\cup V_\beta$ be a standard splitting of $L(p,q)$ into genus 1 handlebodies $V_\alpha,V_\beta$ and denote a meridian disk of $V_\alpha$ (resp.\ $V_\beta$) by $D_\alpha$ (resp.\ $D_\beta$). As originally indicated in \cite{Ber}, Berge knots in lens spaces may be represented as the union of an arc in $D_\alpha$ and an arc in $D_\beta$. They are \emph{simple knots} in the language of knot Floer homology. Further, any simple knot in $L(p,q)$ that has an integer $S^3$ surgery, has a doubly primitive dual in $S^3$.

 A generalization of grid diagram representations of knots in $S^3$ exists for knots in $L(p,q)$, and (similar to the situation in $S^3$) there is a relationship to Legendrians in $(L(p,q),\xi_{p,q})$. Following the language of \cite{BG} and \cite{BGH} if a $n$ is the minimum complexity of a toroidal grid diagram representing $K\subset L(p,q)$ we write $\text{GN}(K)=n$ (see Section \ref{sec:bg} on toroidal grid diagrams). A knot $K$ is simple when $\text{GN}(K)=1$. Thus the Berge Conjecture may be restated: If a knot $K\subset L(p,q)$ has an integer $S^3$ surgery then GN$(K)=1$.

Denote the maximal Thurston-Bennequn number of Legendrian representatives of $K$ by $\tbbar(K)$. It is a consequence of work in \cite{DP} that in the standard contact structure on $S^3$, there is an equality $-\text{GN}(K) = \tbbar(K)+\tbbar(m(K))$, where $m(K)$ is the mirror of $K$. It is possible that their work could be extended to the lens space setting, with the understanding that $\tbbar(K)$ is computed in $\xi_{p,q}$ and $\tbbar(m(K))$ in $\xi_{p,q'}$ where $qq'\equiv -1\md p$.

Hence the Berge Conjecture can be traded for a question about classical invariants of duals to lens space knots in $L(p,q)$. As briefly mentioned above, duals to lens space knots are rationally fibered and support a tight contact structure. But if a knot $K$ supports a given (tight) contact structure, there are Legendrian representatives of $K$ with Thurston-Bennequn number at least -1. Hence we would like to understand the contact structures supported by surgery duals in lens spaces.
}

Theorem \ref{mainthm} has a consequence for fractional Dehn twist coeficients of open books supported by lens space knots.
\begin{cor}If $K\subset S^3$ is a lens space knot with a surgery dual that is not homologous to the dual of a torus knot, then $c(h)<\frac{2}{2g(K)-1}$, where $g(K)$ is the Seifert genus and $h$ is the monodromy for $K$.
\label{FDTCcor}
\end{cor}

We should remark that the bound of Corollary \ref{FDTCcor} (in fact, a slightly better one) can be obtained by using \cite[Theorem 4.5]{KR} in place of Theorem \ref{mainthm} in the proof (see Section \ref{sec:proof}).

In Section \ref{sec:bg} we review work of Baker, Etnyre, and Van Horn-Morris, which allows us to approach Theorem \ref{mainthm} through the rational Bennequin-Eliashberg inequality. We also review how to compute the self-linking number of transverse knots in $\xi_{p,q}$ and what is known about homology classes of the surgery duals in $H_1(L(p,q))$. Section \ref{sec:types} addresses the sharpness of the rational Bennequin-Eliashberg inequality for Berge knots in each of the different types. The proofs of Theorem \ref{mainthm} and Corollary \ref{FDTCcor} are given in Section \ref{sec:proof}. 

\subsection*{Acknowledgments}
The author is very grateful to Ken Baker and Eli Grigsby for a number of motivating and clarifying conversations. This research was partially supported by a CIRGET postdoctoral fellowship.

\section{Background}
\label{sec:bg}

\subsection{Rational open book decompositions and lens space surgeries}
Let $Y$ be a closed, oriented 3-manifold and $K\subset Y$ an oriented rationally null-homologous knot (we consider only the connected case). The genus $g(K)$ is the minimal genus of a rational Seifert surface. If $K$ is the surgery dual to $K'\subset Y'$ then $g(K)=g(K')$. Set $-\chi(K) = 2g(K)-1$.

We call $K$ \emph{rationally fibered} if, for an open regular neighborhood $\nu(K)$, there is a fibration $\theta:Y\setminus\nu(K)\to S^1$ with compact fibers $F$ such that $\bd F$ is not meridional. That is, choosing oriented curves $\set{m,l}\subset \bd(Y\setminus\nu(K))$ with $m$ bounding a disk cooriented with $K$, and $l$ isotopic to $K$ in $\nu(K)$, then $[\bd F]=r[m]+s[l]$ in $H_1(\bd(Y\setminus\nu(K)))$ for some $s\ne 0$. A \emph{rational open book decomposition} of $Y$ is the data $(K,\theta,F)$ for a rationally fibered $K$. The binding of the rational open book is $K$; note, the fiber $F$ is a rational Seifert surface of $K$. When $s=1$ (so $K$ is null-homologous), we call $(K,\theta,F)$ an honest open book decomposition. Subsequently we suppress $\theta$ from the notation.

Alexander showed every closed, oriented $Y^3$ has an honest open book decomposition \cite{Alex}. Given an open book decomposition $(K,F)$ of $Y$, Thurston and Winkelnkemper \cite{TW} construct a contact 1-form $\alpha$ on $Y$ with the properties:
	\begin{itemize}
		\item[(a)]$\alpha(v)>0$ for any $v$ tangent to $K$ (and coherently oriented with $K$);
		\item[(b)]$d\alpha$ is a volume form on the interior of $F$.
	\end{itemize}
Whenever $\xi=\ker\alpha$ for a contact 1-form $\alpha$ satisfying (a) and (b), we say $(K,F)$ \emph{supports} $\xi$. Recently the construction of such a contact 1-form (and hence a supported contact structure) was generalized to rational open book decompositions \cite{BEVHM}. Two contact structures $\xi_1,\xi_2$ supported by $(K,F)$ are isotopic. So given a rational open book decomposition $(K,F)$ of $Y$ we may write $\xi(K,F)$ for the contact structure on $Y$ that it supports.

When $Y=S^3$ Hedden investigated conditions to determine whether a fibered $K$ supports the tight contact structure. He found the answer is connected to a knot being strongly quasipositive, and to the Ozsv\'ath-Szab\'o concordance invariant $\tau$. 
\comment{
We recall the definitions of a strongly quasipositive knot and of the $\tau$ invariant in order to state Hedden's theorem.

Let $\sigma_1,\sigma_2,\ldots,\sigma_{n-1}$ denote the standard generators of the braid group $B_n$ on $n$ strands. For any $i<j\le n$, let $\sigma_{i,j}$ define the braid $\sigma_{i,j}=(\sigma_i\ldots\sigma_{j-2})\sigma_{j-1}(\sigma_i\ldots\sigma_{j-2})^{-1}$ in $B_n$. A knot $K$ is called \emph{strongly quasipositive} if $K$ may be realized as the closure $\hat\beta$ of the braid
	\[\beta=\prod_{k=1}^m\sigma_{i_k,j_k}.\]

We note that the maximum Euler characteristic of a strongly quasipositive knot is known. Suppose $K$ is such a knot, realized as the closure of some $\beta\in B_n$ as above, where $\beta$ is the product of $m$ elements $\sigma_{i,j}$. Then $K$ has a Seifert surface $\Sigma$ constructed from $n$ parallel disks with $m$ half-twisted bands attached (the one corresponding to $\sigma_{i,j}$ being attached to the $i^{th}$ and $j^{th}$ parallel disk). Now, from the ``Legendrianization'' of the closure of $\beta$ (see \cite{Hed1}) which we call $L$, one sees that
	\[tb(L)+\abs{rot(L)}=-n+m.\]
So the Bennequin inequality is sharp on $K$ and the Seifert surface $\Sigma$ has maximum Euler characteristic.

Let $\cl F(Y,K,i)$ denote the filtration that a null-homologous knot $K\subset Y$ induces on the Ozsv\'ath-Szab\'o chain complex $\widehat{CF}(Y,\mathfrak s)$ associated to $Y$ and $\mathfrak s$, a spin$^c$-structure on $Y$. In the setting where $Y=S^3$, the concordance invariant is defined by
	\[\tau(K)=\text{min}\setn{j\in\Z}{i_*:H_*(\cl F(Y,K,j))\to\widehat{HF}(S^3)\cong\Z\text{ is non-trivial.}}\]
}

\begin{thm}[\cite{Hed1}]Let $(K,F)\subset S^3$ be a fibered knot. Then
	\[(K,F)\text{ is strongly quasipositive}\iff \xi(K,F)=\xi_{st}\iff \tau(K)=g(K),\]
where $\xi_{st}$ is the standard tight contact structure on $S^3$, $g(K)$ is the Seifert genus, and $\tau(K)$ the Ozsv\'ath-Szab\'o concordance invariant.
\label{thm:Hed}
\end{thm}

\begin{rem} The method of Hedden's proof uses that $\xi(K,F)=\xi_{st}\iff c(\xi(K,F))\ne0$, where $c(\xi)$ is the Ozsv\'ath-Szab\'o contact invariant, which equivalence uses the uniqueness of the tight contact structure on $S^3$.
\end{rem}

Note that $K$ is transverse in $\xi(K,F)$. Representing $K$ as the closure of a strongly quasipositive braid $\beta$ (for which a natural ``Legendrianization'' $L$ exists), with fiber $F$, the Bennequin inequality $\slk(K) = \tb(L)+|\rot(L)|\le-\chi(F)$ is found to be equality (see \cite{Hed1}). Hence, if a fibered knot in $S^3$ supports the tight contact structure $\xi_{st}$ then by Theorem \ref{thm:Hed} the Bennequin inequality is sharp for that knot. 

In fact, this is a more general phenomenon. If $K$ is rationally fibered with fiber $F$, the construction of $\xi(K,F)$ guarantees that $sl_{\xi(K,F)}(K)=-\chi(F)$. 
The converse question is considered in \cite{BE} (in \cite{EVHM} for the case of an honest open book), when there are no overtwisted disks in the complement of $K$.

\begin{thm}[\cite{BE}] Let $(K,F)$ be a rationally fibered, transverse link in a contact 3-manifold $(Y,\xi)$ such that $\xi$ restricted to the exterior of $K$ is tight. Suppose that $[K]$ has order $r$ in $H_1(Y)$. Then $r\cdot\slk_{\xi}(K)=-\chi(F)$ if and only if either $\xi=\xi(K,F)$ or is obtained from $\xi(K,F)$ by adding Giroux torsion along incompressible tori in the complement of $K$. 
\label{thm:BE}
\end{thm}

Our approach to Theorem \ref{mainthm} involves showing that, except for duals to torus knots, a Berge knot $K\subset L(p,q)$ cannot satisfy the equality in Theorem \ref{thm:BE} when $\xi=\xi_{p,q}$.

In \cite{OS2}, Ozsv\'ath and Szab\'o show that any lens space knot $K$ has the property $\tau(K)=g(K)$. Thus, since all such knots are fibered by the work of Ni \cite{Ni}, Theorem \ref{thm:Hed} above tells us that every lens space knot supports the tight contact structure on $S^3$. 

Any dual $K\subset L(p,q)$ to a lens space knot is rationally fibered (and the order of $K$ equals $p$). Since the dual to $K$ supports $\xi_{st}$ (which has non-trivial contact invariant) and $p>2g(K)-1$ by \cite{KMOS}, the principal theorem of \cite{HPla} implies that $K$ supports a tight contact structure on $L(p,q)$.

\subsection{Classical invariants of links in $(L(p,q),\xi_{p,q})$}
\label{UTClassicalinvts}

Recall the covering map $\pi:S^3\to L(p,q)$ from Section \ref{sec:intro}. For a point $(z_1,z_2)\in \C^2$, represent $z_i$ for $i=1,2$ as $(r_i,\theta_i)$, letting $r_i$ be the modulus and $0\le\theta_i<2\pi$ the argument. Define a 1-form $\alpha_{st} = r_1^2d\theta_1 + r_2^2d\theta_2$ on $S^3\subset\C^2$. Then $\xi_{st}=\ker\alpha_{st}$ gives the standard contact structure on $S^3$. The ($\Z/p\Z$)-action fixes $r_1$ and $r_2$ and adds a constant to $\theta_1$ and $\theta_2$, hence $\xi_{st}$ has a well-defined pushforward $\xi_{p,q}=\pi_*(\xi_{st})$. So defined $\xi_{p,q}$ is a contact structure on $L(p,q)$ which by construction is universally tight.

By Theorem \ref{thm:BE}, for any rationally fibered $K\subset L(p,q)$ which has order $p$ in homology, and in $\xi_{p,q}$ is a transverse knot, $p\ \slk_{\xi_{p,q}}(K,F)=-\chi(F)$ if and only if $\xi_{p,q}=\xi(K,F)$. In this section we discuss calculating $\slk_{\xi_{p,q}}(K)$. Our approach is to consider a (positive) transverse knot in $(L(p,q),\xi_{p,q})$ as the push-off of a Legendrian knot, and to calculate the self-linking number from the invariants of that Legendrian. These classical invariants can be computed through the \emph{toroidal front projection} as we now review, following the discussion in \cite{BG}.

Using coordinates as above, and noting that $r_1$ determines $r_2$ we can identify points in $L(p,q)$ with those of a fundamental domain of the $(\Z/p\Z)$-action,
	\[\setn{(r_1,\theta_1,\theta_2)}{r_1\in[0,1],\ \theta_1\in\left[0,2\pi\right),\ \theta_2\in\left[0,\frac{2\pi}p\right)}.\]
Set $V_\alpha=\{r_1\le1/\sqrt2\}$ and $V_\beta = \{r_1\ge1/\sqrt2\}$ and let $\Sigma$ be the Heegaard torus $V_\alpha\cap V_\beta$, oriented as the boundary of $V_\alpha$ (so $\bd_{r_1}$ is the outward conormal).

Given a Legendrian $L$ in $(L(p,q),\xi_{p,q})$ (which may generically be assumed to miss the circles $\{r_1=0\}$ and $\{r_2=0\}$), the \emph{toroidal front projection} is the projection $f:(r_1,\theta_1,\theta_2)\mapsto (1/\sqrt2,\theta_1,\theta_2)$. Similar to the classical case, we may recover $L$ from $f(L)$. If we set $m=\frac{r_1^2}{(r_1^2-1)}$ then at each point in $f(L)$ we have $m = \frac{d\theta_2}{d\theta_1}$ and so $r_1=\sqrt{\frac{m}{m-1}}$ is determined by the slope of $f(L)$ in $\Sigma$. Also similar to the classical picture, if two arcs in $f(L)$ meet transversally at a point, then the arc with more negative slope (note $m\in(-\infty,0)$) has greater $r_1$ coordinate and so is the arc that passes over, from the viewpoint of $\bd_{r_1}$.

Up to Legendrian isotopy we may assume that each cusp of $f(L)$ is semi-cubical and locally symmetric about a line of slope -1. Thus, in a neighbourhood of a cusp $x$ there is exactly one strand of $f(L)-x$ with tangent slopes in the interval (-1,0). When $L$ is oriented we say $x$ is an \emph{upward cusp} if $\theta_1\vert_{f(L)}$ is increasing on that strand. We call $x$ a \emph{downward cusp} otherwise.

Note that
	\[\setn{(r_1,\theta_1,\theta_2)}{r_1=\frac1{\sqrt2}, \theta_1\in[0,2\pi), \theta_2\in[0,2\pi/p)}\]
is a fundamental domain for $\Sigma$ in our coordinates above. Dropping the constant $r_1$-coordinate, specify an identification to a torus by $(\theta_1,2\pi/p)\sim (\theta_1-2\pi q/p,0)$ and $(0,\theta_2) \sim (2\pi,\theta_2)$. As in \cite[Section4.1]{BG} we identify $\Sigma$ with the underlying torus of a genus 1 Heegard diagram for $L(p,q)$ (which is part of the data in a \emph{twisted toroidal grid diagram}). The identification matches an $\alpha$ curve in the diagram with a curve having a constant $\theta_2$ coordinate, say $\theta_2=0$ (oriented with increasing $\theta_1$), and a $\beta$ curve with one that has constant $\theta_1$ coordinate modulo $2\pi/p$. That is, the $\beta$ curve consists of points $\{(2\pi k/p,\theta_2) \mid k=0,\ldots,p-1\}$ (oriented with increasing $\theta_2$), which is connected since $p,q$ are coprime.

Let $A$ be the core of $V_\alpha$ and $B$ the core of $V_\beta$ oriented to intersect positively. Given any knot $K$, define $a_K,b_K\in\Z/p\Z$ so that $[K]=a_K[A]$ and $[K]=b_K[B]$ in $H_1(L(p,q))$. Write $\langle x,y \rangle$ for the algebraic intersection number of a curve $x$ with curve $y$ on $\Sigma$. For example, $\langle\alpha,\beta\rangle=p$. One can verify that
	\begin{equation}
	a_L \equiv \langle \alpha,f(L) \rangle\md p \qquad \text{and} \qquad b_L \equiv \langle f(L),\beta \rangle \md p.
	\label{ClassByIntNumbers}
	\end{equation}

\begin{rem}The correspondence between Legendrian links in $(L(p,q),\xi_{p,q})$ and \emph{toroidal grid diagrams} was carefully laid out by Baker and Grigsby. In Section \ref{secI} we only need this correspondence for some remarks on Berge duals to torus knots. As we do not need grid diagrams for our main result, we refer the interested reader to \cite{BG}.
\end{rem}

The classical invariants of $L$ can be readily calculated from $f(L)$, as in the standard contact structure on $\R^3$. This was discussed in \cite{Cor}, where the formulae were found by comparing with the $\Z/p\Z$-symmetric link in $S^3$ that covers $L$ via $\pi:S^3=L(1,0)\to L(p,q)$. 

In \cite{Cor} the correspondence between a Legendrian link and a toroidal grid diagram was view with the opposite orientation. However, here the front projection $f(L)$ for a Legendrian $L$ in $(L(p,q),\xi_{p,q})$ can be approximated by a rectilinear projection (a union of arcs parallel to $\alpha$ and to $\beta$) having the property that, at transverse double points, the arc which is parallel to $\beta$ passes over the arc parallel to $\alpha$. Accounting for this convention change, the following theorem was proved in \cite[Section 3]{Cor}.

\begin{thm}[\cite{Cor}]
If $L\subset (L(p,q),\xi_{p,q})$ is an oriented Legendrian link, let $w$ be the writhe of $f(L)$, let $c$ be the number of cusps, $c_d$ the number of downward cusps, and $c_u$ the number of upward cusps in $f(L)$. Also set $a=\langle \alpha,f(L) \rangle$ and $b=\langle f(L),\beta \rangle$. Then
	\al{
	\emph{tb}(L)	&= w-\frac12c+\frac{ab}p \\
	\intertext{and}
	\emph{rot}(L)	&= \frac12\left(c_d-c_u\right)+\frac{(a+b)}p.
	}
\label{tbandrotForm}
\end{thm}

Denote by $L_+$ the positive transverse push-off of $L$. Then the self-linking number satisfies
	\begin{equation} \slk(L_+) = \tb(L) - \rot(L).
	\end{equation}
With the notation of Theorem \ref{tbandrotForm}, we have the following corollary.

\begin{cor}
For any Legendrian $L$ in $(L(p,q),\xi_{p,q})$ we have
	\[\emph{sl}(L_+) = w - c_d + \frac{ab - a-b}p.\]
In particular, we have the equation $p\cdot\emph{sl}(L_+) \equiv ab - a - b \md p$. 
\label{ResofSL}
\end{cor}

Note in Corollary \ref{ResofSL} that, modulo $p$, we have $p\ \slk(L_+)$ only depending on the homology type of the knot. So for the maximal self-linking number in a (primitive) knot type $K\subset L(p,q)$, we have $p\ \slbar(K)\equiv a_Kb_K - a_K - b_K \md p$.

\subsection{Homology classes of duals in $L(p,q)$ to lens space knots}

Conditions on the homology class in $L(p,q)$ of a knot which is dual to a lens space knot in $S^3$ were discussed in \cite[Section 2]{Ras} (we are using the opposite orientation, as done in \cite{Gre}). We recall some of the results here. Let $q'$ be such that $qq' \equiv 1\md p$.

\begin{lem}Given any knot $K\subset L(p,q)$, $b_K \equiv a_Kq \md p$.
\label{HomClassCongr}
\end{lem}
\begin{proof} Expressing $[A]\in H_1(L(p,q))$ in terms of $[B]$ we have $[A] = q[B]$. Hence $b_K[B] = [K] = a_K[A] = a_Kq[B]$. (Alternatively, the equivalence is implied by (\ref{ClassByIntNumbers}) and that $f(L)$ is a closed curve.)\end{proof}

\begin{lem}If $K\subset L(p,q)$ is any knot with positive integral surgery yielding $S^3$ then $a_K^2q \equiv -1 \md p$ and $b_K \equiv -a_K^{-1} \md p$.
\label{QuadResLem}
\end{lem}
\begin{proof}This argument was given in \cite[Section 2]{Ras}. Roughly, the surgery condition requires the (topological) self-linking number $[K]\cdot[K] \equiv 1/p \md 1$. Since $[B]\cdot[B]=-q'/p$, we get $-a_K^2q \equiv -b_K^2q' \equiv 1\md p$ and also $b_K \equiv a_Kq \equiv -a_K^{-1}\md p$ by Lemma \ref{HomClassCongr}.\end{proof}

\begin{lem}For any Legendrian representative $L$ in $(L(p,q),\xi_{p,q})$ of a Berge knot,
		\[p\ \overline{\emph{sl}}(L_+) \equiv -1-a_L+a_L^{-1} \md p\].
\label{pSLMainCong}
\end{lem}
\begin{proof}
	The lemma follows directly from Corollary \ref{ResofSL} and Lemmas \ref{HomClassCongr} and \ref{QuadResLem}.
\end{proof}

Homology classes of Berge knots $K\subset L(p,q)$ in Types I through X satisfy the list of congruences found in \cite[Section 1.2]{Gre}. We will refer to these congruences in the course of the next section. The $k$ in these tables is the homology class ($\pm a_K^{\pm1}$ in our notation). Note that, by Lemma \ref{pSLMainCong}, the congruence class of $-\chi(K)$ agrees with $p\ \slbar(K)$ when $a_K=k$ if and only if they agree for $a_K=-k^{-1}$. Hence it suffices to check the cases $a_K=\pm k$.

\section{Types of Berge knots and Bennequn inequality}
\label{sec:types}
This section is devoted to the proof of the following lemma.

\begin{lem} Let $K\subset L(p,q)$ be a Berge knot. If the congruence
			\begin{equation}
			-1-a_K+a_K^{-1}\equiv-\chi(K) \md p
			\label{Noncongruence}
			\end{equation}
holds then $K$ is dual to a torus knot.
\label{lem:Noncongruence}
\end{lem}

To prove Lemma \ref{lem:Noncongruence} we consider each of the types of Berge knots. For completeness the first subsection deals with Type I, the torus knots. The results there help with the Type II knots in Section \ref{secII}. After this we address Types III{--}VI (Section \ref{secIII-VI}), Type VII (Section \ref{secVII}), Type VIII (Section \ref{secVIII}), and then Types I{--}X (Section \ref{secIX-X}).

\subsection{Type I, torus knots}
\label{secI}

For coprime integers $i,k>1$, the $(i,k)$ torus knot has two lens space surgeries, with coefficient $ik\pm1$. Letting $K\subset L(p,q)$ be the dual Berge knot, $a_K\in\{\pm k,\pm i\}$ \cite{Gre}.

Suppose $p=ik+1$. Noting $i\equiv -k^{-1}$ and the discussion after Lemma \ref{pSLMainCong}, we need only consider $a_K\in\{\pm k\}$. If $a_K=k$ then by Lemma \ref{pSLMainCong}
		\[p\ \slbar(K) \equiv ik - k - i \md p,\]
which equals $-\chi(K)$. Let us show the congruence is equality.

\begin{thm}Let $K\subset L(ik+1,-i^2)$ be the dual to $(ik+1)$-surgery on the $(i,k)$ torus knot, oriented so that $a_K=k$. Then $K$ supports $\xi_{p,q}$.
\label{torusknots}
\end{thm}

\begin{proof}
We use a fact concerning the correspondence between front projections and twisted toroidal grid diagrams in \cite{BG}. If $f(L)$ corresponds to a grid diagram with grid number 1, $f(L)$ can be approximated by the union of two linear segments on $\Sigma$, one parallel to $\alpha$, the other parallel to $\beta$. Suppose the segment parallel to $\alpha$ is in $\{\theta_2=n\}$ for a constant $n\in[0,2\pi/p)$. Then there is an $L'$ Legendrian isotopic to $L$ such that $f(L')$ is approximated by replacing the segment parallel to $\alpha$ by its complement in $\{\theta_2=n\}$. Then, $f(L)$ has $c$ cusps for $c\in\{0,2\}$ and $f(L')$ has $|c-2|$ cusps. In similar fashion, we can replace the segment parallel to $\beta$ by its complement. 

The Legendrian isotopy in the above observation can be achieved through a flow in the radial direction -- relying on $\bd_{r_1}$ being contained in $\xi_{st}=\ker(r_1^2d\theta_1+r_2^2d\theta_2)$. For details see \cite{BG}.

Every Berge knot in $L(p,q)$ has a Legendrian representative $L$ in $\xi_{p,q}$ so that $f(L)$ corresponds to a grid diagram with grid number 1 (\cite{BGH} or \cite{Gre}, grid number 1 was implicit in \cite{Ber}). Since any front projection associated to a grid number one diagram is made of only two arcs, one parallel to $\alpha$ and one parallel to $\beta$, there are 4 front projections associated to a grid number 1 diagram and they are determined by the signs of $\langle \alpha,f(L) \rangle$ and $\langle f(L),\beta \rangle$. Choose $f(L)$ so that both quantities are positive. In this case $f(L)$ has exactly two upward cusps.

\begin{lem}Let $p=ik+1$, $i,k>1$ and coprime. Given an integer $0<m<k$ let $0\le n < p$ be such that $-mi^2\equiv n \md p$. Then $i\le n$.
\label{lem:torusdualwrithe}
\end{lem}
\begin{proof} If not, there is an $m$ such that $-mi^2\equiv n \md p$ and $0\le n<i$. Multiplying by $k$, we have $mi\equiv nk \md p$. But $0<mi<p$ and $0\le nk < p$, so $mi=nk$. Since $m$ is too small to contain all factors of $k$, $k$ and $i$ are not coprime, a contradiction. \end{proof}

Returning to the Berge knot $K\subset L(ik+1,q)$, Lemma \ref{QuadResLem} implies $q\equiv -a_K^{-2} \md p$. So when $a_K=k$ we have $q\equiv -i^2$.

Let $\gamma$ be the oriented segment parallel to $\beta$ in the front $f(L)$ that represents $K$. We chose $\gamma$ so that $0<\langle\alpha,f(L)\rangle=k \le p$ (note that $\langle f(L),\beta \rangle = i$ also). Starting at the initial point of $\gamma$, every subarc that intersects $\alpha$ in $m$ points has endpoint that is distance $\frac{2(p-n)\pi}p$ to the left of the initial point of $\gamma$, where $n \equiv -mi^2 \md p$. Thus by Lemma \ref{lem:torusdualwrithe} there is a $\theta_1$ interval of length $i\big(\frac{2\pi}{p}\big)$ to the left of the initial point of $\gamma$ not intersecting $\gamma$. Our choice of $f(L)$ makes the segment parallel to $\alpha$ pass through this interval, hence the writhe of $f(L)$ is zero.

Since the writhe of $f(L)$ is zero and it has only upward cusps, $p\ \slk(L_+) = ki-k-i =-\chi(K)$. Appealing to Theorem \ref{thm:BE} this proves the result.
\end{proof}

We now discuss the other cases.

For $p=ik+1$ and $a_K=-k$, then $p\ \slbar(K) \equiv -1 + k + i \md p$. Thus if (\ref{Noncongruence}) holds then \begin{equation}2k+2i \equiv 0 \md p.\label{TorusCong1}\end{equation} Since $2k+2i < ik+1$ when $\min(i,k)\ge 4$, without loss of generality we may assume $i<k$ and $i=2$ or $i=3$. If $i=2$ then (\ref{TorusCong1}) implies $0\equiv 2k+4 \equiv 3 \md{2k+1}$, which does not hold. If $i=3$ then (\ref{TorusCong1}) implies $0\equiv 2k+6 \equiv 5-k \md{3k+1}$. This does hold when $k = 5$; however, a direct calculation of the self-linking number (from the front projection) in this case and applying \cite[Corollary 1.6]{Cor} shows that $p\ \slbar(K) = -9 < -\chi(K)$.

If $p=ik-1$  and $a_K=k$, then we have $p\ \slbar(K) \equiv ik - k + i - 2 \md p$ since $1 \equiv ik \md p$. If (\ref{Noncongruence}) were to hold then $2i \equiv 2 \md p$, but this is clearly impossible.  If $a_K = -k$ then $p\ \slbar(K) \equiv -1 + k - i \md p$ and (\ref{Noncongruence}) holds unless $2k \equiv 2 \md p$, which likewise is impossible.

\subsection{Type II, 2-cables of torus knots}
\label{secII}

The knots in this family are determined by integers $i,k\ge 4$ such that gcd$(i,k)=2$. The construction of the Berge knot in $S^3$ is that of a 2-cable (more precisely, the $(2,(\frac k2-1)i+i\pm1)$-cable) of the $(\frac k2,\frac i2)$ torus knot. This 2-cable (or its mirror) is the closure of a positive braid with index $k$ and $(k-2)i+i\pm 1$ crossings. Hence, for $K$ determined by the pair $(i,k)$ in Type II, $-\chi(K)=(ki-k-i)\pm 1$.

The surgery coefficient is $p=ik\pm 1$. As in Type I, we must check $a_K=\pm k$. Note that $p\ \slbar(K)$ is in the same class mod $p$ as the Type I knot with equivalent $a_K$. 

Let $p=ik+1$. If $a_K = k$ we saw in Section \ref{secI} that $p\ \slbar(K) \equiv ki - k - i \md p$. So $p\ \slbar(K) \equiv -\chi(K)\mp1 \md p$ and (\ref{Noncongruence}) cannot hold. If $a_K=-k$ then $p\ \slbar(K) \equiv ki + k + i \md p$ and so (\ref{Noncongruence}) requires $2k+2i \equiv \pm1 \md p$. As $\min(i,k)\ge 4$ we see that $p=ik+1>2k+2i+1$ contradicts (\ref{Noncongruence}) also.

When $p=ik-1$ and $a_K=k$ then $p\ \slbar(K) \equiv ki-k+i-2 \md p$. Then (\ref{Noncongruence}) only holds if either $2i \equiv 3 \md p$ or $2i \equiv 1 \md p$, both of which are ruled out by $\min(i,k)\ge 4$. The case $a_K=-k$ is similar.

\subsection{Types III{--}VI}
\label{secIII-VI}
Berge knots in Types III{--}VI were discussed in \cite{Bg2}, later in \cite{Yam1}. We follow notation of \cite{Yam1} with each knot $K(\delta,\ve,A,k,t)\subset S^3$ being determined by a quintuple of numbers, where $\delta,\ve\in\set{1,-1}$, $A\in\Z_{>0}$, $k\in\Z_{\ge0}$, and $t\in\Z$. Type VI is a subcase of V; we simply discuss Types III{--}V. 

Let $K\subset L(p,q)$ be the dual to $K(\delta,\ve,A,k,t)$. Defining $B,b$ as in Table \ref{BbinTypes}, $(A,B,\delta,b)$ are the coefficients of the basis of the homology of the genus 2 splitting, used by Berge in \cite{Bg2}, and so $p=\abs{Bb+A\delta}$ and $a_K \equiv \pm B \md p$. As a note of comparison to the tabulation in \cite{Ras}, here $p=\abs{\delta A(1-\ve cB)-\delta\ve tB^2}$ where $c=2$ in Type III and $c=1$ in Type IV and V. The ``$d$'' in Rasmussen's table is $A$ in our notation. When $a_K = \pm B$, by Lemma \ref{pSLMainCong} we have 
	\begin{equation}
	p\ \slbar(K) \equiv -1 \mp B \pm B^{-1} \md p
	\label{SLtypeIII}
	\end{equation}

We calculate $-\chi(K)=-\chi(K(\delta,\ve,A,k,t))$. Every $K(\delta,\ve,A,k,t)$, or its mirror, is the closure of a positive braid as follows. Let $\sigma_1,\sigma_2,\ldots,\sigma_{n-1}$ denote the standard set of generators of the braid group $B_n$, and define $W(n)=\sigma_{n-1}\sigma_{n-2}\cdots\sigma_1\in B_n$. For each of Type III{--}V there are two numbers $B>0$ and $b\in\Z$, determined by $(\delta,\ve,A,k,t)$ as indicated in Table \ref{BbinTypes}. 

The knot $K(\delta,\ve,A,k,t)$ is the braid closure of $W(B)^bW(A+1-a)^{\delta}$ where $a\in\set{0,1}$ is fixed in each Type as indicated in Table \ref{BbinTypes} (\cite{Bg2}, see \cite[Lemma 2.1]{Yam1}). If $b\ge0$ and $\delta=1$ then this is a positive braid. If $b<0$ then changing the sign of $\delta$ also changes the sign of $b$. Once this is done, if $\delta=-1$ then use \cite[Lemma 2.3]{Yam1} which shows that, for $m>0$, the closures of $W(n_1)^mW(n_2)^{-1}$ and $W(n_1)^{m-1}W(n_1-n_2+1)$ are isotopic. 

\begin{table}[ht]
\begin{tabular}{c | c | c | c | c}
Type	& $a$	& $A$			& $B$			& $b$\\
\hline
III		& 0		&$\ge2$			& $A(3+2k)-\ve$	&$-\delta\ve(2A+tB)$\\
IV		& 1		&$\ge5$, odd	& $(A(5+2k)-\ve)/2$	&$-\delta\ve(A+tB)$\\
V		& 1		&$\ge3$, odd	& $Al+\ve$		&$-\delta\ve(A+tB)$\\
\hline
\end{tabular}
\caption{A table of the ranges of $A$, and values of $B$ and $b$ in Types III{--}V. Here $\delta,\ve\in\{\pm1\}$, $k\in\Z_{\ge0}$, $t\in\Z$. In Type V, $l=2+k$ if $\ve=1$ and $l=3+k$ if $\ve=-1$.}
\label{BbinTypes}
\end{table}

In any of the cases above, $K(\delta,\ve,A,k,t)$ is the closure of an index $B$ braid and we obtain
	\begin{equation}
	-\chi(K) = \abs{b}(B-1)+\delta(A-a)-B.
	\label{Eulerchar}
	\end{equation}

Note that $A>1$ and $B>2A$ for Types III{--}V, and thus $b$ is never zero.
\begin{lem} Given $B$ and $p$ as above in some Type III{--}V, if $t>0$ then $p \ge tB^2+B-A$ and if $t < 0$ then $p \ge |t+1|B^2+B-A$.
\label{B^2lem}
\end{lem}
\begin{proof}
Define $c\in\{1,2\}$ as above. In each of the three Types we have the inequalities $B \ge (c+1)A-1$ and $A \ge 2$ . If $t>0$ then $p=\abs{Bb+\delta A}= cAB+tB^2\pm A > tB^2 + B - A$ as $A\ge 2$. 

If $t<0$ then $p \ge -tB^2-cAB - A$. As $B - A + 1 \ge cA$ we have
	\[-tB^2 -cAB \ge (-t-1)B^2 + AB - B \ge (-t-1)B^2+B.\]
Thus $p \ge (-t-1)B^2+B-A = |t+1|B^2+B-A$.
\end{proof}

\begin{lem} If $t$ is odd then $p$ is odd.
\label{Parityofp}
\end{lem}
\begin{proof}Note that $A,B$ have opposite parity in Type III, in the other Types, $A$ is odd (with $B$ odd in Type VI). From this the statement of the lemma is readily checked.
\end{proof}

\begin{thm}If $K\subset L(p,q)$ is a Berge knot of Type III{--}V, dual to $K(\delta,\ve,A,k,t)$, and $t\ne -1,0$ then $K$ does not support $\xi_{p,q}$.
\label{thm:typeIII}
\end{thm}
\begin{proof}Our task is to show that (\ref{Noncongruence}) fails; we divide the proof into four cases, determined by $a_K=\pm B$ and whether $b$ is positive or negative. Applying (\ref{SLtypeIII}) and (\ref{Eulerchar}), one finds that if (\ref{Noncongruence}) holds then $|b|(B-1) + \delta A - \delta a + 1 \equiv B^{-1} \md p$ in the case $a_K=B$, and that $|b|(1-B) - \delta (A-a) + 2B - 1 \equiv B^{-1}\md p$ in the case $a_K=-B$. Using that $p=|Bb + \delta A|$ we may manipulate each of these to see that (\ref{Noncongruence}) is equivalent to the congruences in Table \ref{IIIbPos}.  Recall that $B\ge (c+1)A-1$ and $A\ge 2$.

\begin{table}[ht]
	\begin{tabular}{r r} 
	$(1-\delta a)B+\delta A \equiv 1 \md p$		&	$a_K=B,\ b>0$ \\
	$2\delta AB+(1-\delta a)B-\delta A \equiv 1 \md p$	&	$a_K=B,\ b<0$ \\
	$2B^2 - \delta A + (\delta a-1)B \equiv 1 \md p$	&	$a_K=-B,\ b>0$ \\ 
	$2B^2 - 2\delta AB + \delta A + (\delta a - 1)B \equiv 1 \md p$	&	$a_K=-B,\ b<0$ \\
	\end{tabular}
	\caption{Congruences by case, equivalent to (\ref{Noncongruence})}
	\label{IIIbPos}
\end{table}

Suppose that $a_K=B$ and $b>0$. Since $t\ne 0,-1$, we may use Lemma \ref{B^2lem} to find that
		\[ (1-\delta a)B+\delta A \le 2B - A < B^2 + B-A \le p. \] 
The left side is at least $A\ge2$ always, hence the first row of Table \ref{IIIbPos} cannot hold, therefore neither can (\ref{Noncongruence}) when $a_K=B,\ b>0$.

Turning to the case $a_K=B,\ b<0$, since we have $A-a\le B-cA+(1-a)$ and Lemma \ref{B^2lem},
		\[1< (2A-a)B - A+B \le B^2 -cAB +AB+(1-a)B + B - A \le B^2 + B - A \le p.\]
This shows that setting $\delta=1$ in row 2 of Table \ref{IIIbPos} does not hold. When $\delta = -1$ use Lemma \ref{B^2lem} again and note that 
		\[ -p +1 < -B^2 + A +(1+a)B < -B(2A) + A + (1+a)B < 0.\]
Thus the equivalence again fails, and so (\ref{Noncongruence}) fails for $a_K = B,\ b<0$.

The cases with $a_K = -B$ require some work for $|t|$ small. However, when $a_K=-B$, writing $\Theta_b$ for either of the expressions left of the congruence in Table \ref{IIIbPos}, depending on the sign of $b$, observe that $1<\Theta_b$ (it is at least $(B-A)(2B-1)$). Also \begin{equation}\Theta_b<2B^2\text{ if }b>0,\quad\text{ and }\quad\Theta_b < 2B^2+2AB < 3B^2\text{ if }b<0.\label{ThetabBound}\end{equation} So by Lemma \ref{B^2lem} the equivalences in Table \ref{IIIbPos} cannot hold if $t\ge 3$ or $t\le-4$. Furthermore, if $b>0$ then we may also exclude $t=2$ and $t=-3$. 

To address the remaining values of $t$ it is valuable to consider parity. Noting that $A$ is odd if $a=1$ and the parity of $B$ is opposite that of $A$ if $a=0$, we may conclude that $\Theta_b$ is always odd. If $t$ is odd then $p$ is odd, by Lemma \ref{Parityofp}, and so if $t$ is odd and $\Theta_b = mp + 1$, then $m$ must be even. In particular, for $t$ odd, the congruence in question does not hold if $\Theta_b\le 2p$.

When $t=1$ and $b>0$, or $t=-3$ and $b<0$, Lemma \ref{B^2lem} and (\ref{ThetabBound}) imply that $\Theta_b < 2p$. When $t=1$ and $b<0$ then $2p = 2B^2+2cAB - 2\delta A$. This is larger than $2B^2+2AB > \Theta_b$ when $\delta=-1$ and larger than $2B^2 > \Theta_b$ when $\delta=1$; hence $\Theta_b < 2p$ when $b<0$ and $t=1$. Having ruled out the case $t=-3,\ b>0$, we see that rows 3 and 4 of Table \ref{IIIbPos} fail for $t=1$ and $t=3$. 

Now suppose that $t$ is even ($t=\pm2$ being the remaining cases). If $t=2$ then we only need consider $b<0$. By Lemma \ref{B^2lem}, $2p>4B^2>\Theta_b$ and so the congruence of Table \ref{IIIbPos} holds only if $\Theta_b = p+1$. As $\Theta_b < 2B^2 < p$ when $\delta=1$ we assume $\delta=-1$. Consulting the definition of $b$, having $b<0$, $t>0$, and $\delta<0$ forces $\ve=-1$. In that case, $2A<B-1$. Now $p=2B^2+cAB+A$ which is larger than $\Theta_b$ in Type III (where $c=2$). In Types IV and V, where $A\ge 3$, we have
			\[\Theta_b - p = AB - 2A - 2B > AB -3B + 1 \ge 1\]
and thus $\Theta_b \ne p+1$, so Table \ref{IIIbPos} does not hold if $t=2$.

Finally, consider the case $t=-2$. If $b>0$ then $p = 2B^2 - cAB + \delta A$ and we check that
	\[\Theta_b - p = cAB + (\delta a-1)B - 2\delta A \le B^2 - AB + B + (\delta a-1)B - 2\delta A < B^2 +1 < p+1\]
by Lemma \ref{B^2lem}. Since $3A-1\le B$ when $c=2$ and $A\ge 3$ when $c=1=a$ the left side above is positive, so Table \ref{IIIbPos} cannot hold for $b>0$ in this case.

If $b<0$ then $p= 2B^2 - cAB - \delta A \ge B^2 + AB - B - \delta A$. In this case, Table \ref{IIIbPos} requires
		\begin{equation}\Theta_b - p = cAB - 2\delta AB + 2\delta A + (\delta a - 1)B \equiv 1 \md p.
		\label{tcaseminus2}
		\end{equation}
In Type III we have $c=2$. Then $\delta=1$ would make $\Theta_b - p=2A - B < 0$, but $-p+1 < 2A - B$, so the equivalence is not satisfied. If $\delta=-1$ instead, then $1 < \Theta_b-p$ and $\Theta_b-2p = 6AB - 2B^2 - 3A -B$, which is negative for $B\ge3A-1$. This shows (\ref{tcaseminus2}) doesn't hold in Type III.

In Types IV and V (where $c=1$) note the left side of (\ref{tcaseminus2}) is largest when $\delta=-1$ and, as $A\le B-A+1$, this at most equals $B^2+AB-2A-B<p$ since $c=1$. If $\delta=1$ the left side of (\ref{tcaseminus2}) is larger than $-p+1$, so we must have equality
		\[AB - 2\delta AB + 2\delta A + (\delta-1)B = 1.\]
But this equality does not hold for $\delta=\pm1$. Thus (\ref{tcaseminus2}) does not hold for Types IV or V, and we conclude that Table \ref{IIIbPos} does not hold in this case.
\end{proof}

To conclude that (\ref{Noncongruence}) does not hold for Types III {--} V, and thus that none of these knots support $\xi_{p,q}$, it remains to check $t=0$ and $t=-1$. We rule out these cases in Appendix \ref{sec:appendix}, but remark on one case. When $t=-1$, there is an instance of a dual to a torus knot and the equality $p\ \slbar(K) = -\chi(K)$ holds. This occurs for (the dual to) $K(\delta,\ve,A,k,t) = K(1,1,2,0,-1)$, which is in Type III and has $b = 1$ (this is the only instance when $b=1$ since $b=-\delta\ve(cA+tB)$, $B\ge(c+1)A-1$ and $A\ge 2$). As discussed above, the knot in $S^3$ is in this case the closure of $W(B)^bW(A+1-a)^\delta=W(5)W(3)$, which is a stabilization of $(\sigma_2\sigma_1)^2$, which has closure a trefoil.

\subsection{Type VII, knots on fiber surface of trefoil}
\label{secVII}

In \cite{Ras} a knot $K\subset L(p,q)$ in Type VII is described by the condition $a_K^2 + a_K + 1 \equiv 0 \md p$ on its homology class. We discuss the relation to Berge's original description.

\begin{lem} Suppose $a,p$ are integers with $0<a<p$. There is a pair of integers $r,s$ with \emph{gcd}$(r,s)=1$ such that $p=r^2+rs+s^2$ and $a\equiv r^2s^{-2} \md p$ if and only if $a^2+a+1\equiv 0 \md p$.
\label{lem:quadform}
\end{lem}
\begin{proof} For the ``only if'' direction, since $a^3-1\equiv s^{-6}(r^6-s^6)=s^{-6}(r-s)(r^3+s^3)(r^2+rs+s^2)$, we see that $a^3\equiv 1 \md p$. A similar calculation shows that $r^3s^{-3}\equiv 1 \md p$, and so $a^2\equiv rs^{-1} \md p$. Hence, 
	\al{
	a+a^2+1	&\equiv r^2s^{-2} + rs^{-1} +1\\
			&\equiv s^{-2}\left(r^2+rs+s^2\right).
	}
Now, for any $0<a<p$, suppose that $a^2+a+1=mp$, for some integer $m$. Then $mp$ is primitively represented by the positive definite quadratic form $f(x,y)=x^2+xy+y^2$. A number is so represented if and only if its prime decomposition is of the form
		\[mp=3^{\ve}\prod p_i^k,\]
where $\ve$ is 0 or 1 and each $p_i$ is a prime congruent to 1 mod 3 (see, for example, \cite[Chapter 3 (p176)]{NZM}). Since $p$ must then have a prime decomposition of the same form, $p$ is also primitively represented by $f$. That is, there are coprime $r,s$ such that $r^2+rs+s^2=p$. 

Solutions to $x^2+x+1\equiv 0\md p$ are in one-to-one correspondence with solutions to $(x')^2\equiv-3\md p$ via the relation $x'\equiv2x+1\md p$. The number of such $x'$, which is $2^n$ where $n$ is the number of distinct prime factors $p_i\equiv1\md3$ of $p$, equals the number of primitive representations $p=f(r,s)$ (restricting to $r,s>0$). Each primitive representation $(r,s)$ gives rise to the solution $x=r^2s^{-2}$, and so $a$ must arise from one of these.
\end{proof}

The $r$ and $s$ of Lemma \ref{lem:quadform} were originally used to describe these Berge knots in $S^3$. We use this to find $-\chi(K)$.

The right and left handed trefoils can be embedded with Seifert surface $\Sigma$ on a standard genus 2 Heegaard splitting of $S^3$; see Figure \ref{fig:dblPrim}, where the curves $g_1$ and $g_2$ are cores of the two handles of $\Sigma$. A knot embedded on $\Sigma$ 
can be described by a pair of coprime numbers $r,s$: each Berge knot in Type VII is homologous to $r[g_1]+s[g_2]$ on $\Sigma$. Let $K'$ be such a knot, with $[K']=r[g_1]+s[g_2]$ on $\Sigma$. The surgery framing $p$ given by $\Sigma$ was found to be $p=r^2+rs+s^2$ \cite{Ber}.

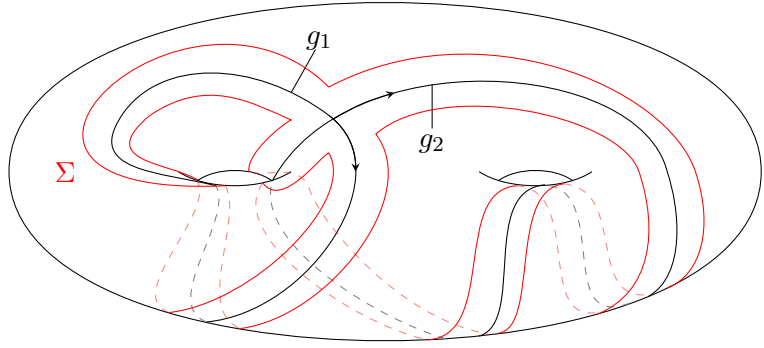
\begin{figure}[ht]
	\begin{center}
	\begin{tikzpicture}[scale=0.25,>=stealth]
		\draw 	(-20,0) ..controls (-20,12) and (20,12).. (20,0)
						..controls (20,-12) and (-20,-12)..(-20,0);
		\draw	(-11,0) ..controls (-9.5,-1) and (-6.5,-1) ..(-5,0)
				(5,0) ..controls (6.5,-1) and (9.5,-1) .. (11,0);
		\draw	(-10,-0.5) ..controls (-9,0.25) and (-7,0.25) .. (-6,-0.5)
				(6,-0.5) ..controls (7,0.25) and (9,0.25) ..(10,-0.5);
		\draw	(-3,3) ..controls (1,0) and (-4,-7.5)..(-9.5,-8);
		\draw[->]	(-3,3) ..controls (-1.4,1.8) and (-1.57,.15) ..(-1.57,0);
		\draw[dashed,gray]	(-9.5,-8)..controls (-12,-7) and (-8,-2.5) ..(-9,-0.75);
		\draw	(-9,-0.75)..controls (-12,0) and (-16,0) .. (-14,3)
						..controls (-12,6) and (-7,6) ..(-3,3);
		\draw[->]	(-2.75,2.8) ..controls (-1.25,3.65) ..(.5,4.2);
		
		\draw	(-6,-0.5)..controls (-2,7) and (13.5,6) ..(15,0)
						 ..controls (15.67,-2) and (16,-5.5) ..(14,-6.64);
		\draw[dashed,gray]	(13,-7)..controls (9,-8.5) and (12,-1) ..(8.5,-0.7);
		\draw	(8.5,-0.7)..controls (5,-1) and (7.5,-8) ..(5,-8.75);
		\draw[dashed,gray]	(3.25,-8.9)..controls(0,-8) and (-7,-4) ..(-6,-0.5);
		\draw	(-3.5,7) node {$g_1$}
				(2.5,1.5) node {$g_2$};
		\draw[thin]	(-3.7,6.5) -- (-5,4.2)
					(2.5,2.3) -- (2.5,4.65);
		
		\draw[red]	(-3,1) ..controls (-1.5,-1.5) and (-7,-7) ..(-11.5,-7.5)
					(-9.95,-0.45) ..controls (-11.5,0.5) and (-15.5,0.5) ..(-12.5,3)
								  ..controls (-9.5,5) and (-7.25,3.75) ..(-5,2.75)
								  ..controls (-5.5,2.5) and (-7.35,1)..(-7.25,0)
					(-3,1) ..controls (-3.5,0.75) and (-5.5,-2) ..(-6.5,-0.6)
					(2.25,-8.95) ..controls (6,-8.75) and (3.5,-1) ..(7,-0.75)
					(6,-8.6) ..controls (8.25,-8.35) and (6.5,-1) ..(9.5,-0.65)
					(11.25,-7.6) ..controls (14.25,-6.4) and (14.5,-3) ..(13.5,0)
								 ..controls (12.5,3) and (3,5) ..(-0.5,2)
					(15,-6.2)	..controls (16.8,-5.5) and (17.5,-3) ..(16.5,0)
								..controls (15,4.5) and (5,8.75) ..(-3,4.5)
					(-0.5,2)	..controls (2,-1.25) and (-3,-7.75) ..(-7.7,-8.35)
					(-3,4.5)	..controls (-10.5,12) and (-25,-1.5) ..(-8.5,-0.75)
					(-17,0) node {$\Sigma$};
		
	\draw[red,dashed,opacity=0.5]	(-11.5,-7.5) ..controls (-14,-6.5) and (-9,-2.25) ..(-10,-0.5)
									(-7.1,-0.05) ..controls (-6,-0.25) and (-4.5,0.5) ..(-4,-1)
												 ..controls (-3.5,-2.5) and (0.5,-7.75) ..(5.5,-8.65)
									(-6.5,-0.6)  ..controls (-8.5,-3.25) and (-2,-8) ..(2.25,-8.95)
									(7,-0.75)	..controls (10.5,-1) and (8,-8) ..(11.25,-7.5)
									(9.5,-0.65) ..controls (13,-1) and (10.5,-7.25) ..(14.3,-6.5)
									(-8.5,-0.75)..controls (-7.5,-2.5) and (-10.2,-7.35) ..(-7.7,-8.35);
	
	\end{tikzpicture}
	\end{center}		
	\caption{Berge knots on a Seifert surface of the trefoil}
	\label{fig:dblPrim}
\end{figure}

Writing $K$ for the surgery dual of $K'$, we find $-\chi(K)=-\chi(K')$ by realizing $K'$ (or its mirror) as the closure of a positive braid in $S^3$ which has index $r+s$ and $rs+r^2-r+s^2-s$ crossings \cite{Bak}. Therefore 
		\[-\chi(K)=r^2+rs+s^2-2(r+s) \equiv -2(r+s) \md p.\] 

Thus, (\ref{Noncongruence}) in this case is
				\[-1-a_K+a_K^{-1} \equiv -2(r+s) \md p.\]
By recalling Lemma \ref{lem:quadform}, $a_K^{-1} \equiv a_K^2 \md p$ and $a_K^2+a_K+1\equiv0 \md p$. Thus (\ref{Noncongruence}) is equivalent to
			\[2a_K^2 \equiv -2(r+s) \md p.\] 
As $a_K\equiv r^2s^{-2}\md p$, this implies $2rs^{-1} \equiv -2(r+s) \md p$. Since $p$ is not even, this is equivalent to 
		\[r \equiv -s(r+s)\equiv r^2 \md p.\]
But $r$ is invertible and so this implies that $r\equiv1$ and so $r=1$ as $r<p$. Considering the description of $K'$ on $\Sigma$, if $r=1$ then $K'$ can be isotoped onto a standard torus.

\subsection{Type VIII, knots on fiber surface of figure-eight}
\label{secVIII}
Similar to the trefoil, the figure eight knot and a Seifert surface of it $\Sigma$ may be embedded on the genus 2 Heegaard surface. The Berge knots of Type VIII are the embedded closed curves on $\Sigma$, determined as before by a coprime pair $r,s$. In this case the surgery coefficient is $p=|s^2-rs-r^2|$ and $a_K=\pm rs^{-1}$.

\begin{lem}Let $K\subset L(p,q)$ be a Type VIII Berge knot determined by $r,s>0$. Then either $p\ \slbar(K)\equiv 0 \md p$ or $p\ \slbar(K)\equiv -2 \md p$.
\label{VIIISLequivclass}
\end{lem}
\begin{proof}
	Let $a\equiv rs^{-1} \md p$. By Lemma \ref{pSLMainCong} we have $p\ \slbar(K) \equiv -1-a_K+a_K^{-1} \equiv -1\mp a\pm a^{-1} \md p$. Note that
					\[a^2+a-1 \equiv s^{-2}(r^2+rs-s^2) \equiv 0 \md p,\]
	which implies that $a^{-1}\equiv a+1 \md p$. So if $a_K=a$ then $p\ \slbar(K) \equiv 0 \md p$ and if $a_K=-a$ then $p\ \slbar(K)\equiv -2 \md p$.
\end{proof}

We now indicate how to narrow our focus.

\begin{lem}Suppose $a,b > 0$ are relatively prime and let $n=b^2-ab-a^2$. Then there exist relatively prime $c,d>0$ such that $c>d$, $c^2+cd-d^2 = |n|$, and $cd^{-1} \equiv ab^{-1} \md n$.
\label{abSwitch}
\end{lem}
\begin{proof}
	Write $f(x,y)=x^2+xy-y^2$. Defining $\Gamma(x,y)=(y-x,x)$, let $(a_i,b_i)=\Gamma^i(a,b)$ and note that
				\begin{equation}
				f(a_1,b_1) = -f(a,b),
				\label{pUnchanged}
				\end{equation}
	for any $a,b$. Hence $f(a_i,b_i)=(-1)^if(a,b)=(-1)^{i+1}n$. Note that for any $x,y>0$, if $f(x,y)<0$ then $x<y$.

	We first suppose that $n>0$ (implying that $a<b$). Since $f(a_{2i},b_{2i})=-n<0$ for any $i$, we must have $a_{2i+1}=b_{2i}-a_{2i}>0$ provided $a_{2i},b_{2i}$ are both positive. Moreover, if $a_{2i-1}<b_{2i-1}$ then $b_{2i+1}>0$. We will show that there is some $i\ge 0$ such that $a_{2i+1}>b_{2i+1}$. Choosing such an $i$ minimally then guarantees $c=a_{2i+1}$ and $d=b_{2i+1}$ are both positive, and $f(a_{2i+1},b_{2i+1})=n$. 

	To this end, observe that $f(a_1,b_1)$ being positive implies that $b_1<\frac{1+\sqrt5}2a_1$ which is equivalent to $a<\frac{\sqrt5-1}2b$. Let $F_i$ denote the $i^{th}$ Fibonacci number, starting the sequence by $F_{-1}=0, F_0=1$. By induction argument on $i$, it follows that $a_i=(-1)^i(F_ia-F_{i-1}b)$ (which implies $b_i=(-1)^{i-1}(F_{i-1}a-F_{i-2}b)$ for $i\ge1$).

	Given $i\ge1$, the supposition that $a_{2i-1}<b_{2i-1}$ implies by the previous paragraph that $a> \frac{F_{2i}}{F_{2i+1}}b$. But this latter inequality cannot hold for all $i\ge1$ since $a<\frac{\sqrt5-1}2b$ and $\lim_{i\to\infty}\frac{F_i}{F_{i+1}}=\frac{\sqrt5-1}2$. Thus $c=a_{2i+1}>b_{2i+1}=d$ for some $i\ge0$.

	Now suppose that $n<0$. If $a>b$ then as $f(a,b)=-n$ we may simply take $c=a$ and $d=b$. Otherwise define $(a_i,b_i)$ as before. An analogous argument to before  (with inequalities reversed) shows that $c=a_{2i}>b_{2i}=d$ for some $i$ and $f(c,d)=-n>0$.
%

	We finish the proof by observing that for any $i$, $a_i(a_i+b_i)\equiv b_i^2 \md n$. Thus 
				\[a_{i+1}b_{i+1}^{-1}\equiv (b_i-a_i)a_i^{-1} \equiv (b_i^2-a_i^2)b_i^{-2} \equiv a_ib_i^{-1} \md n.\]
\end{proof}

As a consequence, every Type VIII Berge knot is obtained from a pair $r>s>0$, with $r,s$ relatively prime. For in a given lens space, there is a unique knot in each homology class that has grid number one. As every Berge knot $K\subset L(p,q)$ has grid number one and the homology class is given by $\pm rs^{-1}$, Lemma \ref{abSwitch} implies that it suffices to consider $r>s>0$ (and thus $p=r^2+rs-s^2$).

\begin{lem}A Berge knot of Type VIII in $L(p,q)$ does not support $\xi_{p,q}$.
\label{VIIIthm}
\end{lem}
\begin{proof}
	For relatively prime $r>s>0$, let $m\ge2$ be such that $(m-1)s<r<ms$. Describing $K$ as the closure of a braid was discussed in \cite[Appendix B]{Bak}, where it is shown that the Type VIII knot (or its mirror) in $S^3$ is the closure of a positive braid with index $s$ and $(ms-r)(r-(m-1)s)+s(s-1)$ twists. Thus we calculate the maximal Euler characteristic by
				\begin{equation}
				-\chi(K)=x(r,s):=(2m-1)rs-m(m-1)s^2-r^2+s^2-2s.
				\label{eqn:chi}
				\end{equation}
	We consider the $m$ in the expression $x(r,s)$ as determined by $r$ and $s$ as in the previous paragraph.

	Note that $(r,s)\mapsto(r-s,s)$ acts as the identity on $x(r,s)$: since $(m-2)s<r-s<(m-1)s$,
				\[x(r-s,s)=(2m-3)(r-s)s - (m-1)(m-2)s^2-(r-s)^2+s^2-2s = x(r,s).\]
	Define $\bar r = r-(m-1)s$. We obtain a pair $(\bar r,s)$, where $\bar r<s$. As $m\ge2$,
			\al{
			-\chi(K) = x(\bar r,s)	&= \bar rs - {\bar r}^2+s^2-2s \\
									&= rs -(m-1)s^2 - r^2 + 2(m-1)rs-(m-1)^2s^2+s^2-2s \\
									&< rs + r^2 -(m-1)s^2 - (m-1)^2s^2 + s^2-2s \\
									&\le rs + r^2 - s^2-2 = p-2.
			}
	As $0<\bar r<s$ it is clear that $0<x(\bar r,s)$. Thus it cannot be that $-\chi(K) \equiv 0 \md p$ nor that $-\chi(K)\equiv-2 \md p$. By Lemma \ref{VIIISLequivclass} the equivalence (\ref{Noncongruence}) is impossible.
\end{proof}

\subsection{Types IX{--}X, sporadic Berge knots}
\label{secIX-X}
Each knot $K\subset L(p,q)$ in these families is determined by an integer $j\in\Z\setminus\{0,-1\}$. The corresponding surgery coefficient, homology class and $-\chi(K)$ is given in Table \ref{typeIXXknots} \cite{Yam2}.

\begin{table}[ht]
	\begin{tabular}{c | c | c | c | c}
	Type	& $p$		& $a_K$		& $-\chi(K),\ j>0$	& $-\chi(K),\ j<-1$	\\ \hline
	IX	& $22j^2+9j+1$	& $\pm(11j+2)$	& $22j^2-1$		& $22j^2+18j+3$	\\ \hline
	X	& $22j^2+13j+2$&$\pm(11j+3)$& $22j^2+4j-1$		& $22j^2+22j+5$ 	\\ \hline
	\end{tabular}
	\caption{Data for knot $K$ corresponding to $j\in\Z\setminus\{0,-1\}$}
	\label{typeIXXknots}
\end{table}

We must check that (\ref{Noncongruence}) fails in each case.

In Type IX, suppose that $a_K=11j+2$. When $j>0$ note that $(11j+2)^{-1}\equiv -(22j+5) \md p$. Hence, (\ref{Noncongruence}) becomes $-8-33j \equiv 22j^2-1 \md p$, which is equivalent to $24j +6 \equiv 0 \md p.$ But $p > 24j+6$ for all $j>0$ so this congruence cannot hold.

When $j<-1$ then (\ref{Noncongruence}) is the same as $-8-33j \equiv 22j^2+18j+3 \md p$, which is equivalent to $42j+10 \equiv 0 \md p.$ Here $p > |42j+10|$ if $j<-2$, and the congruence also fails when $j=-2$; thus (\ref{Noncongruence}) fails in this case also.

If $a_K=-(11j+2)$ then we would need $42j+8 \equiv 0 \md p$ when $j>0$ and $24j+4 \equiv 0 \md p$ when $j<-1$. These also do not hold, showing (\ref{Noncongruence}) is false in this case too.

In Type X, suppose that $a_K = 11j+3$ which has $-(22j+7)$ as an inverse modulo $p$. For $j>0$ a sharp Bennequin bound requires $24j+8 \equiv 0 \md p$ and for $j<-1$ it requires that $42j+14\equiv 0 \md p$. But here $p > 24j+8$ for $j>0$ and $p>|42j+14|$ if $j<-2$, so these congruences do not hold; If $j=-2$ then $p=64$ and $42j+14 = -70 \not\equiv 0 \md{64}$.

Finally, let $a_K = -(11j+3)$. Here, if $j>0$ then the congruence is $42j+12 \equiv 0 \md p$, and if $j<-1$ the congruence is $24j+6 \equiv 0 \md p$. These again fail, $p$ being larger than $|24j+6|$ for $j<-1$ and larger than $42j+12$ if $j>1$. When $j=1$ we have $p=37$ and $42j+14 = 56 \not\equiv 0 \md p$. Thus (\ref{Noncongruence}) fails for each $K$ in Type X.

\section{Proof of main result}
\label{sec:proof}

Our main result, Theorem \ref{mainthm}, is a consequence of the following.

\begin{thm}If $K\subset L(p,q)$ is a Berge knot and $K$ is not dual to a torus knot, then $K$ does not support the contact structure $\xi_{p,q}$.
\label{mainthmBerge}
\end{thm}
\begin{proof} By the work done in Section \ref{sec:types} we have that $-1-a_K+a_K^{-1}\not\equiv-\chi(K) \md p$. By Theorem \ref{thm:BE} and Lemma \ref{pSLMainCong} this implies that $K$ cannot support $\xi_{p,q}$.
\end{proof} 

\begin{proof}[Proof of Theorem \ref{mainthm}]
As previously remarked, there are at most two universally tight contact structures on $L(p,q)$, and exactly two if $0<q<p-1$. The other universally tight structure $\overline\xi_{p,q}$ is obtained from $\xi_{p,q}$ by reversing coorientation. Any positively transverse knot $T$ in $\overline\xi_{p,q}$ is the positive transverse pushoff of some oriented Legendrian $L$ in $\overline\xi_{p,q}$. The knot $L$ is Legendrian in $\xi_{p,q}$ and has the same Thurston-Bennequin invariant, but has rotation number which is negative of that in $\overline\xi_{p,q}$. Hence the self-linking number $\slk_{\overline\xi_{p,q}}(T)$ equals $\tb_{\xi_{p,q}}(L)+\rot_{\xi_{p,q}}(L)$, the self-linking of the negative transverse pushoff of $L$ in $\xi_{p,q}$. 

In Section \ref{sec:types} where it was checked that $-1-a_K+a_K^{-1}\not\equiv-\chi(K) \md p$, for a $K$ a Berge dual, we allowed for either sign on $a_K$. This amounts to checking the self-linking number of the negative pushoff. Hence, if a Berge knot $K\subset L(p,q)$ supports either $\xi_{p,q}$ or $\overline\xi_{p,q}$, then it is dual to a torus knot. 

For any dual $K\subset L(p,q)$ to a lens space knot, there is a Berge knot in the homology class of $K$ with isomorphic knot Floer homology \cite{Gre}. As the knot Floer homology determines $-\chi(K)$ and $p\ \slbar$ modulo $p$ only depends on the homology class of $K$, Theorem \ref{mainthm} then follows.
\end{proof}

\begin{proof}[Proof of Corollary \ref{FDTCcor}]
Let $K'\subset L(p,q)$ be the surgery dual to $K$. As the exteriors of $K$ and $K'$ are homeomorphic and $K'$ is primitive, the exterior of $K$ in $S^3$ is finitely covered by the exterior of a knot, say $\tilde{K}\subset S^3$. By \cite[Chapter 3]{GAW}, the covering $S^3\setminus\nu(\tilde{K}) \to S^3\setminus \nu(K)$ is cyclic.

Alternatively, one can take $\Sigma_p(K)$, the $p$-fold cyclic cover branched over $K$, then do 1-surgery on the image of $K$ in $\Sigma_p(K)$ to get $\tilde{K}\subset S^3$ as the surgery dual. Let $h$ as the monodromy of the fibration $S^3\setminus\nu(K)$ and $g$ the monodromy for $S^3\setminus\nu(\tilde{K})$. Letting $\gamma$ be a pushoff of the image of $K$ in $\Sigma_p(K)$ into a page of the fibration, we have that $g = h^p\circ D_{\gamma}^{-1}$, where $D_{\gamma}$ is the right-handed Dehn twist along $\gamma$.

Our assumption on $K$ and Theorem \ref{mainthm} imply that $\tilde{K}$ supports an overtwisted contact structure on $S^3$. By \cite{HKM} this implies \[p c(h) - 1 = c(g) < 1\] and so $c(h)< 2/p < 2/(2g(K)-1)$.
\end{proof}


\appendix
\section{Types III{--}V, for $t=0,-1$}
\label{sec:appendix}

We finish the two remaining cases of $t=0$ and $t=-1$, with the same notation as in Section \ref{secIII-VI}. In each case we need to show the congruence in Table \ref{IIIbPos} does not hold.

We begin with $t=0$, in which case $p=cAB+\delta A$ if $b>0$ and $p=cAB-\delta A$ if $b<0$. 

Now $(1-\delta a)B+\delta A < cAB+\delta A=p$ when $b>0$, and so row 1 of Table \ref{IIIbPos} fails when $t=0$, $a_K=B$, and $b>0$. If $b<0$ and $a_K=B$ then Table \ref{IIIbPos} has $\delta(2A-a)B-\delta A + B \equiv 1 \md p$. Since $p=cAB-\delta A$, we can rewrite the congruence as $(3-c-\delta)A - (\delta a-1)B \equiv 1 \md p$. When $c=2$ the left side is $B+(1-\delta)A$ and when $c=1$ it is $(1-\delta)B+(2-\delta)A$. Both are less than $p$ but larger than 1.

If $a_K=-B$ then we need to show that $\Theta_b\not\equiv 1 \md p$. For $t=0$, the sign of $b$ is $-\delta\ve$, and so $cAB\equiv \ve A$ in this case. By describing $B$ in terms of the parameter $k$ (see Table \ref{BbinTypes}) one can then show that
		\al{
		2B^2 \equiv \begin{cases} -\ve B+1, \text{ in Type III} \\ \ve B+1,\text{ in Type IV} \\ 4\ve B-2,\text{ in Type V},\end{cases}
		}
from which we obtain, if $b>0$, that
		\al{
		\Theta_b+\delta A=2B^2 + (\delta a-1)B 	&\equiv \begin{cases} - (\ve+1)B + 1,\text{ in Type III} \\
																	  - B + 1,\text{ in Type IV} \\
																	  (3\ve-1)B - 2,\text{ in Type V}.\end{cases}
		}
In the case $b<0$ one need only change the left side above to $\Theta_b-\delta A$ and subtract either $A$ or $2A$ from the right-hand side.

In Type III we see that $\Theta_b\not\equiv 1\md p$ since $\Theta_b$ is equivalent to either $A-2B+1$ or $-A+1$ which are strictly between $-p+1$ and $0$ as $p=2AB+\delta A$. The Type IV case is similar as $\Theta_b \equiv -B -\delta A + 1 \md p$. In Type V, $\Theta_b$ is either congruent to $2B+A-2$ if $\ve=1$, or to $-4B-A-2$ if $\ve=-1$. As $A\ge 3$ and $B>2A$ in this case, the former is strictly between 1 and $p$. As the latter occurs when $\ve=-1$, and $A$ is odd in Type V, $-4B-A-2$ is either $-B-2$ modulo $p$ or is strictly greater than $-p+1$. In each case, $\Theta_b\not\equiv 1\md p$.

The case when $b<0$ can be handled similarly to verify that $\Theta_b\not\equiv 1\md p$.

We now finish by addressing the case $t = -1$. 

As $b=-\delta\ve(cA+tB)$ and $B>cA$ we have that the sign of $b$ is $\delta\ve$ in this case. Thus $p=B^2-cAB+\ve A$. 

Suppose that $c=1$ (working in Types IV and V). If $b<0$, $\delta=1$ then $\Theta_b = 2B^2-2AB+A = 2p+3A$, but $1<3A<AB-A<p$ so $\Theta_b\not\equiv 1\md p$. Now, for $b<0$, let $k\ge 2$. This implies $B\ge 4A+1$. Then if $\delta=-1$ we see that
	\al{
	\Theta_b \le 2B^2 + \frac12(B-1)B - A -2B 	&< 3B^2+B+A \\
												&\le 4B^2-4AB + A \\
												&= 4p-3A.
	}
As we observed before, $t$ being odd implies that $\Theta_b=mp+1$ can only hold if $m$ is even. So for $\Theta_b\equiv 1\md p$ to hold requires $\Theta_b=2p+1$ in this case. But $4AB-2B > 3A+1$ makes $\Theta_b=2p+1$ impossible. 

The cases when $k=0$ or $k=1$ can be handled individually. In the first ($k=0$), since $B=\frac12(5A-1)$ or $B=2A+1$, we have that $\Theta_b =(7A-3)B-A=4p-(A+1)B-5A > 2p+1$, or we have $\Theta_b = 6AB-A = 6p-6B-7A > 4p+1$ when $A\ge5$. When $A=3$ we get $\Theta_b=123 \equiv -1 \md p$. In the second ($k=1$) we have either $B=\frac12(7A-1)$ or $B=3A+1$. In the first case $\Theta_b = (9A-3)B - A = 4p-(A+1)B-5A>2p+1$ and in the second $\Theta_b = 8AB-A=4p-4B-5A>2p+1$. Thus we have taken care of when $b<0$ and $c=1$.

When $b>0$ and $c=1$ then $\Theta_b = 2B^2 - \delta A + (\delta-1)B \equiv 2AB -3\delta A +(\delta-1)B \md p$. Since $B-A > A$ we have
						\[2p = 2(B-A)B+2\delta A > 2AB +2\delta A.\]
Now, as $b>0$ we have $\delta=\ve$. This implies that if $\delta=-1$ then $2B>5A$ and so for $\delta=\pm1$, we get $2AB -3\delta A +(\delta-1)B < 2AB+2\delta A$. As $1 < 2AB -3\delta A +(\delta-1)B$ this impies that $\Theta_b\not\equiv 1\md p$ for $b>0$ and $c=1$.

To finish the case $t=-1$ we consider $c=2$, the Type III knots. Recalling the exceptional case mentioned in Section \ref{secIII-VI}, when $b>0$ we may assume $k\ge 1$ in the case $t=-1$. This implies that $B\ge 5A-1$ and thus that $p\ge 3AB-\ve(B-A)$. Also note that $\Theta_b\equiv 4AB-3\ve A-B \md p$ in this case. Now $1<4AB-3\ve A-B <6AB-2\ve B+\ve A\le 2p$. Since $t$ is odd, this implies that $\Theta_b\not\equiv 1\md p$.

When $t=-1$ and $b<0$, consider first the case that $k\ge 1$. Here the expression $\Theta_b$ is congruent modulo $p$ to $4AB-3\ve A+2\ve AB-B$ which is strictly larger than 1. If this latter expression is also at most $2p$, which occurs when $\ve=-1$, we are done as in the previous paragraph. Thus we may suppose that $\ve=1$ and that 
						\[2p\le 4AB-3\ve A+2\ve AB-B = 6AB - 3A-B\] 
which is less than $4p$ when $k\ge1$. Hence, for $\Theta_b\equiv 1\md p$ to hold it would have to be that $6AB-3A-B = 2p+1$. Either $B=5A-1$ or $B>5A-1$ and in each case $6AB-3A-B < 2p+1$.

Finally, when $b<0$, if $k=0$ then $B=3A-\ve$. We then have that $\Theta_b = 12A^2+6\ve A^2 - 5A - 7\ve A + \ve$ and $p = 3A^2-3\ve A+1$. Thus, if $\Theta_b = mp+1$ for some $m\ge1$ we separate into the cases that $\ve=1$ or $\ve =-1$. If $\ve =-1$ then $\Theta_b=mp+1$ has no real solution for $A$ if $m\ge 2$. Again, $t$ being odd allows us to exclude $m=1$. If $\ve=1$ then $\Theta_b=mp+1$ has no real solution in $A$ if $m\ge7$. In fact, $\Theta_b=mp+1$ does not have an integer solution for $A$ if $m\ge1$ and so $\Theta_b\not\equiv 1\md p$ when $b<0$, finishing the case $t=-1$.

\bibliography{bergedualsrefs}
\bibliographystyle{alpha}
	
\end{document}